\documentclass[oneside,british]{amsart}
\usepackage[T1]{fontenc}
\usepackage[latin9]{inputenc}
\usepackage{amsthm}
\usepackage{amsbsy}
\usepackage{amstext}
\usepackage{amssymb}
\usepackage{graphicx}
\usepackage{esint}

\newcommand{\PoissonHam}{\tilde{\mathsf{H}}}
\newcommand{\PoissonMat}{P}
\newcommand{\Poih}{\tilde{\mathsf{h}}}
\newcommand{\SymplecticMat}{S}
\newcommand{\unscaledp}{\tilde{p}}
\newcommand{\CanonicalHam}{\tilde H}

\makeatletter

\newtheorem{theorem}{Theorem}[section]

\newtheorem{lemma}[theorem]{Lemma}

\theoremstyle{definition}
\newtheorem{definition}[theorem]{Definition}
\newtheorem{remark}{Remark}

\let\oldsqrt\sqrt
\def\sqrt{\mathpalette\DHLhksqrt}
\def\DHLhksqrt#1#2{%
\setbox0=\hbox{$#1\oldsqrt{#2\,}$}\dimen0=\ht0
\advance\dimen0-0.2\ht0
\setbox2=\hbox{\vrule height\ht0 depth -\dimen0}%
{\box0\lower0.4pt\box2}}

\@ifundefined{showcaptionsetup}{}{%
 \PassOptionsToPackage{caption=false}{subfig}}
\usepackage{subfig}

\usepackage{babel}

\title[Symplectic Invariants of the Euler Top]
      {Semi-global symplectic invariants of the Euler top}

\author[George Papadopoulos and Holger Dullin]{}

\subjclass{Primary: 37J35; Secondary: 37J15, 70H06, 70E40.}

 \keywords{Euler top, Picard-Fuchs equation, semiglobal symplectic invariants,
Birkhoff normal form, Liouville integrable.}

\thanks{HRD was supported in part by ARC grant DP110102001.}

\begin{document}
\maketitle

\centerline{\scshape George Papadopoulos and Holger R. Dullin }
\medskip
{\footnotesize
 \centerline{School of Mathematics and Statistics}
   \centerline{The University of Sydney}
   \centerline{Sydney, NSW 2006, Australia}
}

\bigskip

\begin{abstract}
We compute the semi-global symplectic invariants near the
hyperbolic equilibrium points of the Euler top. 
The Birkhoff normal form at the hyperbolic point is computed using Lie series.
The actions near the hyperbolic point are found using Frobenius expansion of its Picard-Fuchs equation.
We show that the Birkhoff normal form can also be found 
by inverting the regular solution of the Picard-Fuchs equation.
Composition of the singular action integral with the Birkhoff normal form 
gives the semi-global symplectic invariant. Finally,
we discuss the convergence of these invariants 
and show that in a neighbourhood of the separatrix the pendulum is not  symplectically equivalent to any Euler top.
\end{abstract}

\section{Introduction}

The Euler top is a classical Hamiltonian system which describes the
rotational motion of a rigid body in free space in the absence of
a potential. In the centre of mass frame this can be considered as
the rotation of the rigid body about a fixed point.
This system has been treated extensively in the literature for the last
300 years, so we only give a few selected references here \cite{Leimanis1965,Marsden1999,Cushman1997}.
After reduction by the $SO(3)$ symmetry the Hamiltonian $\tilde{\mathsf H}$ can be written in terms of
the moment of inertia tensor $ M:=\text{diag}(\Theta_{1},\Theta_{2},\Theta_{3})$
and the angular momentum $\boldsymbol{L}:=(L_{1},L_{2},L_{3})$ in the body frame as
\begin{equation} 
\PoissonHam(\boldsymbol{L})=\frac{1}{2}\boldsymbol{L}^{\text{T}} M^{-1}\boldsymbol{L}=\frac{1}{2}\sum_{i=1}^{3}\frac{L_{i}^{2}}{\Theta_{i}}.\label{eq:PoissonHamiltonian}
\end{equation}

Without loss of generality, we assume the ordering $0<\Theta_{1}<\Theta_{2}<\Theta_{3}$.
The principal moments of inertia satisfy triangle inequalities $\Theta_{1}\le\Theta_{2}+\Theta_{3}$
and so forth cyclically. We exclude the degenerate cases $\Theta_{1}=\Theta_{2}$ or $\Theta_{2}=\Theta_{3}$
since then  there is no hyperbolic equilibrium. This way
of writing the Hamiltonian is not symplectic as the angular momenta
are not canonical variables. Instead the system has a Poisson structure
with Poisson structure matrix
\[
\PoissonMat=\left(\begin{array}{ccc}
0 & -L_{3} & L_{2}\\
L_{3} & 0 & -L_{1}\\
-L_{2} & L_{1} & 0
\end{array}\right)
\]
so that the differential equation is written as
\[
\dot{\boldsymbol{L}}=\PoissonMat \, \nabla \PoissonHam = \boldsymbol{L} \times  \nabla \PoissonHam.
\]

Due to conservation of angular momentum, the dynamics of the Euler
top takes place on spheres of constant magnitude of angular momentum $\Vert\boldsymbol{L}\|=\ell$.
In fact the total angular momentum is a Casimir $  C(\boldsymbol{L}):=\Vert\boldsymbol{L}\|^{2}$
of the Poisson structure, $ \PoissonMat \nabla  C = 0$.
The level set of the Casimir 
\begin{eqnarray*}
\mathcal{C}_{\ell} & := & \left\{ \boldsymbol{L}\in\mathbb{R}^{3}\colon C(\boldsymbol{L})=\ell^{2}\right\} 
\end{eqnarray*}
is a sphere and the level set of the Hamiltonian, the energy surface
\begin{eqnarray*}
\mathcal{E}_{\Poih}&: = & \left\{ \boldsymbol{L}\in\mathbb{R}^{3}\colon\PoissonHam(\boldsymbol{L})=\Poih\right\} 
\end{eqnarray*}
is an ellipsoid.
 The solution curves are given by the intersection $\mathcal{C}_{\ell}\cap\mathcal{E}_{\Poih}$.
The (non-degenerate) intersections of the sphere with the ellipsoid
give concentric ellipse-like curves, centred about elliptic equilibria
of which there are four. There are two separatrices, the intersections
of which occur at the two hyperbolic equilibria. In total there are
six equilibria given by $\pm\ell\hat{\boldsymbol{L}}_{i},\, i\in\{1,2,3\}$
where $\hat{\boldsymbol{L}}_{i}$ is defined as the unit vector along
the $L_{i}$ axis. We are interested in the unstable hyperbolic equilibria
($i=2$), corresponding to a steady rotation about the  principal axis of inertia 
corresponding to the middle moment of inertia $\Theta_{2}$.
The eigenvalues of the linearisation about the hyperbolic equilibria are $\pm \lambda$, where
\begin{equation}
\lambda=\frac{\ell}{\Theta_{2}}\sqrt{\frac{\left(\Theta_{2}-\Theta_{1}\right)\left(\Theta_{3}-\Theta_{2}\right)}{\Theta_{1}\Theta_{3}}}.
\end{equation}

In this work the main aim is to calculate the semi-global symplectic
invariants of the Euler top near the hyperbolic equilibrium point.
In \cite{Dufour1994,Toulet1996}, Dufour, Molino and Toulet introduce
the classification of integrable systems using their semi-global symplectic
invariants. Their approach considers the triple $(\mathcal M,\omega,\mathcal F)$, where $\mathcal M$ is a two-dimensional manifold, 
$\omega$ is a symplectic 2-form, and $\mathcal F$ is a Morse foliation given by the 
levels of a Morse function $F$.
Then the equivalence between two integrable Hamiltonian dynamical systems with one degree of freedom 
is introduced in \cite{Dufour1994} by the following definition:
\begin{definition}
Two triples $(\mathcal M_1,\omega_1,\mathcal F_1)$ and $(\mathcal M_2,\omega_2,\mathcal F_2)$ are said to be equivalent if there exists a symplectomorphism between $\mathcal M_1$ and $\mathcal M_2$ that preserves the foliation.
\end{definition}
In formulas in the analytic case this means there is a symplectic diffeomorphism $\Psi:\mathcal M_1 \to \mathcal M_2$ and a diffeomorphism
$\Xi : \mathbb{R} \to \mathbb{R}$
such that 
\begin{equation}
F_1\circ \Psi = \Xi \circ F_2.
\label{eq:equivalence_formula}
\end{equation}
They introduce the semi-global symplectic invariant as the regular part of the action integral 
measuring the symplectic area near a separatrix (see  Figure \ref{fig:ActionAreas}), 
written in a particular canonical coordinates system defined 
using the Birkhoff normal form at the hyperbolic equilibrium point.
The main theorem in \cite{Dufour1994} then is 
\begin{theorem}
In a full neighbourhood of a separatrix of the same topological type 
two systems are equivalent if and only if their semi-global symplectic 
invariants coincide. \label{thm:equivalence}
\end{theorem}

This theorem motivates us to calculate the semi-global symplectic invariant of the Euler top, the first step that may lead to a non-trivial equivalence between the Euler top
and other Hamiltonian systems with one degree of freedom.
We can, for example, ask whether a particular Euler top (it is a three parameter family) 
is equivalent to the pendulum, for which the invariants have been 
calculated in \cite{Dullin2013}.
In the last section we will show that the answer to this question is no.

The theory of semiglobal symplectic invariants was further developed by San Vu Ngoc in \cite{Ngoc2003} who
extended it to focus-focus equilibria for system of two degrees of freedom.
The first explicit computation of semi-global invariants near a focus-focus point
was done by Dullin in \cite{Dullin2013} for the spherical pendulum.
Hyperbolic-Hyperbolic points were treated in \cite{Dullin2007} where
the C.~Neumann system was used to illustrate the general theory.
Other notions of equivalence have been used to study the Euler top in \cite{Bolsinov2004,Bolsinov1997,Orel1997,Bolsinov1994,Holm1991,Marsden1999}.

\begin{remark}
It should be noted that for $C^{\infty}$ systems the relation \eqref{eq:equivalence_formula}
is not true as such, but may need to be formulated separately in each sector defined by the separatrix.
Furthermore, in general the symplectic invariant which is the Taylor series of the regular part
of the action integral is not necessarily convergent. 
However we need not worry about these two issues, since both the Hamiltonian
and Poisson structure of the Euler top are analytic.
\end{remark}

Plan of the paper:
first we will be calculating the Birkhoff normal form algorithmically using
Lie series. For one degree of freedom, the Birkhoff normal form about
a hyperbolic equilibrium is a series in powers of the regular action
only. Next we show that the action integrals satisfy a Picard-Fuchs
ODE. Then we solve this ODE using the method of Frobenius, and arrive
at series expansions for the action integrals. We  find that if
we invert the regular action integral then we  recover the
Birkhoff normal form. Once we have the Birkhoff normal form, we calculate
the symplectic invariant by composing the singular action integral
with the Birkhoff normal form, and extracting the regular part.
Finally we discuss convergence and non-equivalence to the pendulum.

\section{Calculating the Birkhoff normal form}

In order to calculate the Birkhoff normal form of the Hamiltonian
of the Euler top at the unstable equilibrium, we introduce local canonical
variables $( q,\unscaledp)$ with symplectic structure $\SymplecticMat:=\left(\begin{array}{cc}
0 & 1\\
-1 & 0
\end{array}\right)$:
\begin{lemma}
The Poisson map $\boldsymbol{\Phi}:\mathbb{R}^{3}\rightarrow\mathbb{R}^{2}$
defined by
\begin{subequations}\label{eq:map}
\begin{align}
q& = \mathrm{Arg}(L_2+\mathrm iL_1) \label{eq:map1}\\
\unscaledp&=L_3\label{eq:map2}
\end{align}
\end{subequations}
maps the Hamiltonian $\PoissonHam(\boldsymbol{L})$ of \eqref{eq:PoissonHamiltonian} with value $\Poih$ and Poisson structure
$\PoissonMat$ on the symplectic leaf ${\mathcal C}_\ell$ into the 
standard symplectic structure $\SymplecticMat$ in canonical variables $(q,\unscaledp)$
and Hamiltonian $\CanonicalHam(q,\unscaledp)$ with value $\tilde{h}$ given by
\begin{equation}
 \CanonicalHam(q,\unscaledp)=\frac{1}{2}\left[\unscaledp^{2}\left(\Theta_3^{-1}-f(q)\right)-\ell^{2}\left(\Theta_2^{-1}-f(q)\right)\right]\label{eq:CanonicalHamiltonian}
\end{equation}
where $f(q):=\Theta_1^{-1}\sin^{2}q+\Theta_2^{-1}\cos^{2}q$. 
\end{lemma}
\begin{proof}
Substituting equations (\ref{eq:map}) into \eqref{eq:CanonicalHamiltonian}
and using $\ell^2 = L_1^2 + L_2^2 + L_3^2$
yields the required Hamiltonian \eqref{eq:PoissonHamiltonian} up to the constant $\tfrac{1}{2}\ell^{2}\Theta_2^{-1}$, so that
$\CanonicalHam( \boldsymbol{\Phi}(\boldsymbol{ L})) + \tfrac{1}{2}\ell^{2}\Theta_2^{-1} = \PoissonHam (\boldsymbol{ L})$.
The hyperbolic equilibrium $\ell\hat{\boldsymbol{L}}_2$ is mapped to the origin $(q,\unscaledp)=(0,0)$, 
while   $-\ell\hat{\boldsymbol{L}}_2$ is mapped to $(q,\unscaledp)=(\pi,0)$.
To derive the new symplectic structure, compute the $2\times3$ Jacobian matrix $D \boldsymbol{\Phi}$
and verify that 
${D}\boldsymbol{\Phi}\PoissonMat \left({D}\boldsymbol{\Phi}\right)^{\mathrm{T}}=\SymplecticMat$.
\end{proof}

Note that this transformation to ``cylindrical coordinates for the sphere'' is not defined globally on sphere, 
but only on the punctured sphere with the two points $\pm\ell \hat{\boldsymbol{L}}_3$ (where
$L_1 = L_2 = 0$) removed.
However, the transformation is valid near the unstable equilibria $\pm \ell \hat{\boldsymbol{L}}_2 $ and in a full neighbourhood of their separatrix.
An alternative transformation introduces $L_1$ as momentum instead of $L_3$, 
and the corresponding Hamiltonian $\CanonicalHam(q,\unscaledp)$ has $\Theta_1$ and $\Theta_3$ interchanged. 

At this stage it is convenient to define
the dimensionless real parameters \begin{subequations} 
\begin{align}
\rho & :=\sqrt{\frac{\Theta_{1}(\Theta_{3}-\Theta_{2})}{\Theta_{3}(\Theta_{2}-\Theta_{1})}}\\
\kappa & :=\rho-\rho^{-1},\label{eq:kappa}
\end{align}
\end{subequations}which will be fundamental in the upcoming analyses.
Note that if $\rho\to\rho^{-1}\Leftrightarrow\kappa\to-\kappa$ are
exchanged then $\Theta_{1}\to\Theta_{3}$ are exchanged. The involution
$\rho\to-\rho^{-1}$ leaves $\kappa$ invariant. Furthermore, $\kappa^{2}=\rho^{2}-2+\rho^{-2}$
is rational in the moments of inertia. When restricting to the physical
range $\rho>0$, making $\rho$ the subject in \eqref{eq:kappa} yields
the unique injection $\rho=\frac{1}{2}(\kappa+\sqrt{\kappa^{2}+4})$.
We are then able to re-write our Hamiltonian (originally posed with three parameters)
in terms of a single dimensionless parameter by performing the following non-dimensionalisation:
\begin{lemma}
Using $\frac1{\lambda}$ as units of time, $\ell$ as units of angular momentum,
and $\frac{\ell}{\lambda}$ as units of moment of inertia,
the Hamiltonian in non-dimensional form is
\begin{equation}
         H(q,p) = \frac12 \left ( - p^2 ( \rho + \rho^{-1} \sin^2 q) +  \rho^{-1} \sin^2 q \right). \label{eq:DimensionlessH}
\end{equation}
\end{lemma}
The proof is a simple calculation. The new scaled angular momentum $p=\frac{\unscaledp}{\ell}$ is dimensionless, as is the value of the Hamiltonian 
$h = \frac{\tilde h }{  \lambda \ell}$.
\begin{remark}
Each time a transformation is done, the variables change. For clarity and simplicity of notation we use 
the same letters $(q,p)$ for old and new variables, but it should be noted that each transformation introduces different variables.
In our notation the tilde designates quantities with dimensions, while from this point onwards we use non-dimensionalized 
quantities $H$, $q$, $h$ without tilde. Sans-serif font is used to designate quantities in the original Poisson system,
so that $\Poih  - \frac12 \ell^2 \Theta_2^{-1} = \tilde h  = h \lambda \ell$.
\end{remark}
\begin{remark}
The area form on the original sphere in these variables is $\ell\, \mathrm{d} q \wedge \mathrm{d} \tilde p 
= \ell^2\, \mathrm{d} q \wedge \mathrm{d}p$ so that the scaled symplectic area  $\mathrm{d} q \wedge \mathrm{d}p$ 
differs from the true area on the Casimir sphere $\mathcal{C}_\ell$ by a factor of $\ell^2$.
\end{remark}

The original Hamiltonian has a group of discrete symmetries generated by $L_i \to -L_i$, $i = 1,2, 3$. In the canonical variables these 
correspond to $q \to -q$, $q \to \pi - q$, $p \to - p$, respectively. 
The global analysis we are going to present later is simplest if there is only a single hyperbolic equilibrium on the separatrix in question, 
and therefore we are going to consider the Euler top modulo its discrete symmetry group. 
Any two pairs of the three discrete
symmetries generate the group of symplectic discrete symmetries of the Euler top, 
which is isomorphic to Klein's Vierergruppe $V = \mathbb{Z}_2 \otimes \mathbb{Z}_2$.
In the canonical variables a possible choice of generators is
$S_1(q, p) = (- q, -p)$ and $S_2(q,p) = (\pi + q, p)$ which are both involutions.
A fundamental region for the quotient of the cylindrical $(q,p)$ phase space $[-\pi, \pi) \times (-1,1)$ 
by the group $V$ generated by $S_1$ and $S_2$ can be chosen as  the positive quadrant $(q,p)\in[0,\pi)\times[0,1)$.
 This corresponds to a quarter of the original sphere $$V\left(\mathcal C_{\ell}\right):=\left\{\boldsymbol L\in\mathcal C_{\ell}\colon0<L_1<\ell,-\ell<L_2<\ell,0<L_3<\ell\right\}.$$

We Taylor expand  the non-dimensional Hamiltonian $H(q,p)$
about the origin $(q,p)=(0,0)$ for analysis near the equilibrium $\boldsymbol{L}=\ell\hat{\boldsymbol{L}}_{2}$. 
The quadratic terms are $\tfrac12(\rho^{-1} q^2 - \rho p^2 )$.
The Williamson (linear) normal form of the hyperbolic equilibria is found after a symplectic
linear transformation (e.g. as outlined in \cite{Bolsinov2004}). 
Although the Williamson normal form is unique up to the overall sign of the $qp$ term which we chose to be positive,
the transformation is not; we chose to perform a symplectic scaling $q \to \sqrt{\rho}q,\, p \to \frac p{\sqrt{\rho}}$ followed by a rotation by $-\frac{\pi}4$, 
so that the positive quadrant in the new coordinates corresponds to positive Hamiltonian.
With this convention the Williamson normal form becomes unique.
\begin{lemma}
The symplectic linear transformation
\[
\left(\begin{array}{c}
q\\
p
\end{array}\right)\mapsto
\left(\begin{array}{cc}
\sqrt{\rho} & 0 \\ 0 & \tfrac1{\sqrt{\rho}} 
\end{array}\right)
\frac{1}{\sqrt{2}}
\left(\begin{array}{rr}
1 & 1 \\ - 1 & 1
\end{array}\right)
\left(\begin{array}{c}
q\\
p
\end{array}\right)
\]
puts the quadratic part of the Hamiltonian $H$ into Williamson's normal form, and the new Hamiltonian is
\[
H_* = q p - \frac{1}{8 \rho} (q^2 - p^2)^2 - \frac{\rho} {24 } (q+p)^4 + \mathcal{O}(6)
\]
\end{lemma}

In order to arrive at the Birkhoff normal form, we use the method
of Lie transforms to remove terms that are not powers of $qp$. A recursive algorithm for this is given in, e.g.~\cite{Meyer2010}.
The algorithm is implemented in Mathematica; we state the main results in the following Theorem:
\begin{theorem}
The Birkhoff normal form
 of the Euler top at the hyperbolic equilibrium
point is given by
\begin{eqnarray}
H^{*}(J) & = & J-\frac{\kappa}{4} J^2-\frac{\kappa^{2}+4}{16} J^3-\frac{5\kappa\left(\kappa^{2}+4\right)}{128} J^4 \nonumber \\
 && -\frac{3\left(\kappa^{2}+4\right)\left(11\kappa^{2}+12\right)}{1024} J^5 \label{eq:BNF}
   -\frac{7\kappa\left(\kappa^{2}+4\right)\left(9\kappa^{2}+20\right)}{2048} J^6 \\
 &&     -\frac{\left(\kappa^{2}+4\right)\left(527\kappa^{4}+1776\kappa^{2}+720\right)}{16384} J^7+\mathcal{O}\left(J^8\right)\nonumber
\end{eqnarray}
where $J = q p$ in the new variables.
\end{theorem}
Note that the parameter dependence on the right hand side is only through the dimensionless parameter $\kappa = \rho - \rho^{-1}$
and the power series is in terms of the dimensionless action  $J$.
The normal form coefficients are displayed up to order 14 in the  canonical variables $q$ and $p$.

\section{Action integrals via Picard-Fuchs equation}

\begin{figure} 
{\includegraphics[width=0.47\textwidth]{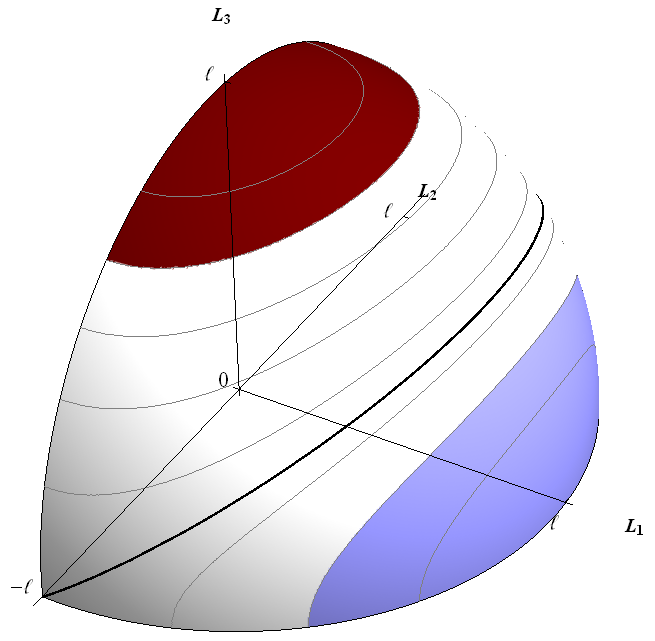}} \hspace{1em}
{\includegraphics[width=0.47\textwidth]{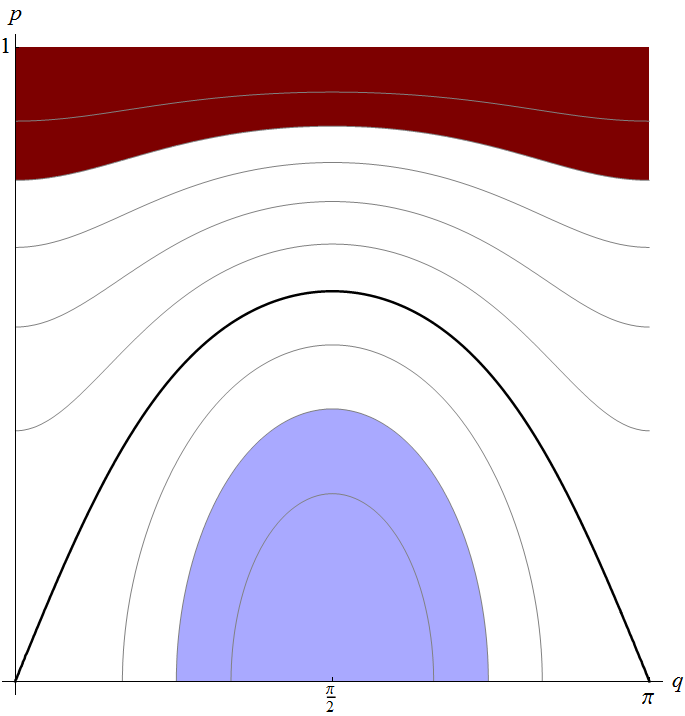}}\\
\hspace{0.7cm} Casimir sphere  \hspace{3cm} Canonical phase plane\\
\caption{ 
Areas enclosed by the same closed orbits after discrete symmetry reduction
(left) on the quarter Casimir sphere $V\left(\mathcal C_\ell\right)$ and 
(right) on the scaled and discrete symmetry reduced canonical phase plane  $(q,p)\in[0,\pi)\times[0,1)$.
The parameter is $\kappa = 0.5$.
The thick black line is the separatrix with $h=0$. 
The lighter shade (blue) below the separatrix shows the action of an orbit with $h>0$ and 
the darker shade  (red) above the separatrix shows the action of an orbit with $h<0$.
The symplectic invariant is calculated from these areas in the singular limit $h \to 0$.
\label{fig:ActionAreas}}
\end{figure}

The essential step in the calculation of the semi-global symplectic invariants 
is the computation of the action integrals, which are given by complete elliptic integrals in the case of the Euler top.
Since we are interested in the series expansions of these integrals the most natural approach 
is not through the integral itself, but instead through the so called Picard-Fuchs ODE that the 
integral satisfies. 
The derivation of the Picard-Fuchs equation proceeds in a way similar to \cite{Dullin2001}.
Frobenius expansions of this linear ODE then gives the desired series.
This  gives a basis for the vector space of solutions of the linear ODE, and in a second 
step the particular solutions corresponding to the action integrals of the Euler top are found.
\begin{figure}
{\includegraphics[width=8cm]{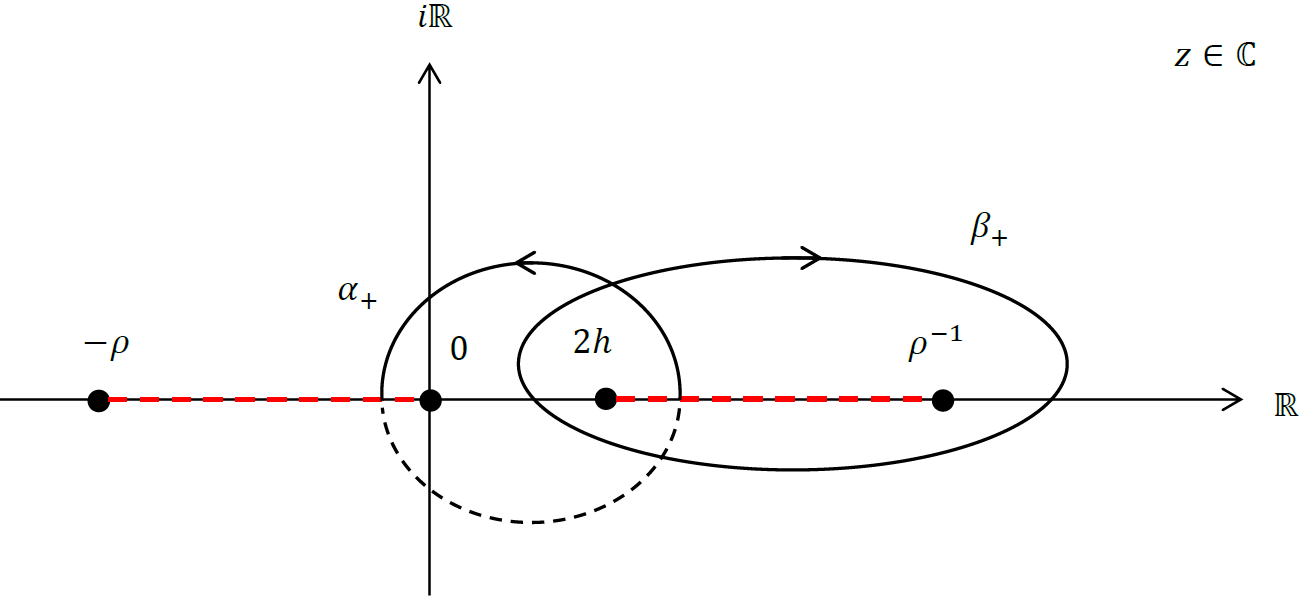}} \\
 {\includegraphics[width=8cm]{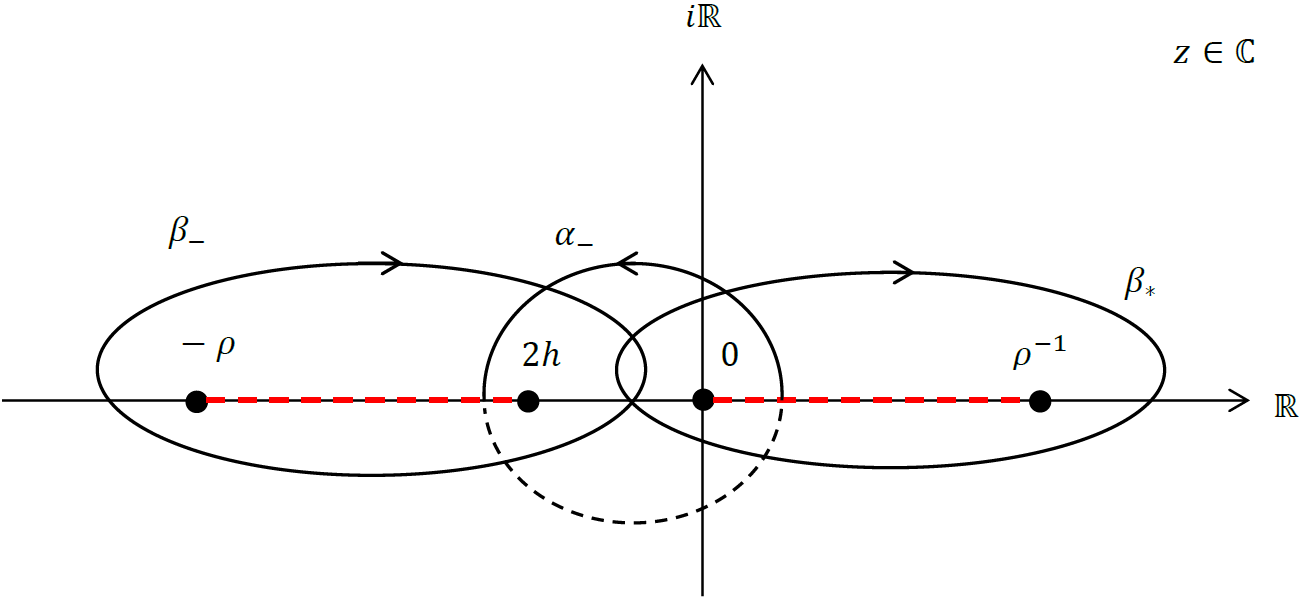}}
\caption{The $\alpha$ and $\beta$ cycles, choices of branch cuts and branch points of $\Gamma$ in the $z\in\mathbb C$ plane
for $ h>0$ (top) and $h<0$ (bottom). \label{fig:contours}}
\end{figure}

\begin{lemma}
The scaled and discrete symmetry reduced action of the Euler top with scaled energy $h$ normalised to 0 at the unstable equilibrium is 
a complete elliptic integral of the third kind on the curve
\[
     \Gamma := \{ (z, u) \in \mathbb{C}^2 \,|\, u^2 = (2 h  - z) w(z)^2,  \quad w(z)^2 = z( z^2 + \kappa z - 1) \}
\]
over the Abelian differential $\zeta$ given by
\[
    I_{\beta_{\pm}}(h)= \frac{1}{4\pi} \oint_{\beta_{\pm}} \zeta(h),\quad \zeta(h):=\frac{\sqrt{2h-z}}{w(z)}\,\mathrm{d}z
\]
along the cycles $\beta_\pm$ for $\pm h > 0$ as specified in Fig.~\ref{fig:contours}.
\end{lemma}
\begin{proof}
We derive the action using the Hamiltonian $H(q,p)$ in canonical variables $(q,p)$. 
Solve $H(q,p)= h$ for $p$ and integrating with respect to $q$ in order to find the 
lighter (blue) area in Fig.~\ref{fig:ActionAreas} gives
\[
I_{\beta_+}(h)
  =  \frac{1}{2\pi}\int_{q_1}^{q_2} \sqrt{-\frac{2 h - \rho^{-1} \sin^2 q}{\rho + \rho^{-1} \sin^2 q } }\,\mathrm{d}q
\]
where $q_1 = \pi/2-q_0$ and $q_2 = \pi/2 + q_0$ are the roots of the numerator of the integrand with  $q_0 = \cos^{-1}( \sqrt{ 2 h \rho})$.
Upon performing the change of integration variable $z= \rho^{-1} \sin^2(q)$ 
we arrive at the required integral
over the differential 1-form $\zeta(h)$ as defined above.
The new denominator $\left[w(z)\right]^{2} = z\left(z^{2}+\kappa z-1\right)=z(z+\rho)(z-\rho^{-1})$
depends only upon the single parameter $\kappa = \rho - \rho^{-1}$.
Note that on $\Gamma$ we find that $\zeta=\frac{u}{w^{2}}\,\mathrm{d}z$ is meromorphic,
so that we can equate the real integral over $z\in\left[q_1,q_2\right]\subset\mathbb{R}$
to an equivalent complex contour integral over $z\in\beta_{+}\subset\mathbb{C}$,
namely
\[
\int_{q_{0}}^{\pi-q_{0}}\zeta = \frac{1}{2}\oint_{\beta_{+}}\zeta \,.
\]
which can be evaluated by continuously shrinking $\beta_{+}$ so that
it sits entirely upon the real axis. 
We find that $\beta_{+}$ gives
two equal contributions as $z$ traverses each side of the branch
cut. Then, $q$ traversing the interval $[q_{1},q_{2}]$ once corresponds
to $z$ traversing the whole closed (shrunken) path $\beta_{+}$ once.
The complex closed loop integral around $\beta_+$ gives twice the real integral along the interval $\left[2h, \rho^{-1}\right]$,
and so altogether $I_{\beta_+} = \frac{1}{4\pi} \oint_{\beta_+} \zeta$.

A slightly more complicated argument applies in the case $h  < 0$.
To get the lighter shaded (red) area in Figure~\ref{fig:ActionAreas} the integrand is $(1-p)\,\mathrm{d}q$ 
instead of just $p\,\mathrm{d} q$
and the real integration interval is $q \in [0, \pi]$. 
Complexifying this integral gives the loop $\beta_*$ so that $I_{\beta_-} = \frac12-\frac{1}{4\pi} \oint_{\beta_*} \zeta$.
On the Riemann sphere we have $\beta_* + \beta_+ + \gamma_\infty = 0$
where $\gamma_\infty$ is a loop around the pole at infinity. Due to the non-zero residue of $\zeta$ at infinity the constant term cancels and we get $I_{\beta_-} = \frac{1}{4\pi} \oint_{\beta_-} \zeta$.
\end{proof}

\begin{remark}
The unscaled action $\tilde I$ as a function of the unscaled and unshifted energy $\Poih$
can be rewritten in the symmetric form 
\[
\tilde I_{\beta_{\pm}}( \Poih, \ell) =
    \frac{1}{2\pi} \oint_{\beta_{\pm}} \sqrt{- \frac{2\Poih - \ell^2  \tilde z}{\left(\tilde z-\Theta_1^{-1}\right)\left(\tilde z-\Theta_2^{-1}\right)\left(\tilde z-\Theta_3^{-1}\right)} }\, \mathrm{d} \tilde z 
       = 2 \ell I_{\beta_\pm}(h)
\,.
\]
The scaled action $I$ depends on $\Poih$ and $\ell$ only through 
$$h =\frac{\Poih-\tfrac 12 \Theta_2^{-1} \ell^2}{ \lambda \ell},$$ 
and on $\Theta_1,\Theta_2,\Theta_3$ only through $\kappa$.
In the transformed variable the roots $\tilde z = \Theta_1^{-1},\Theta_2^{-1},\Theta_3^{-1}$  correspond
to the roots $z=-\rho,0,\rho^{-1}$ of $w$, respectively.
\end{remark}

\subsection{Derivation of the Picard-Fuchs equation}

We now derive the Picard-Fuchs ODE of ${I}(h)$. The Abelian differential $\zeta$ 
lives on the complex  manifold $\Gamma$. By de Rham
cohomology theory, there must exist a relationship between derivatives
of $\zeta$ on $\Gamma$. In fact, we follow the route of Clemens
in \cite{Clemens1980} and find that there exists a linear combination
of the $h$ derivatives of $\zeta$ that equals a total differential. 
However, we cannot exploit the same simplifications as Clemens does
since our independent variable is fixed to be the energy $h$, 
because otherwise we would lose the connection 
to the Birkhoff normal form (see below).
A similar approach was taken in \cite{Dullin2001}.
The relation between the differentials is given in the  following Lemma:
\begin{lemma}
There exists a function $v(z)$ meromorphic on $\Gamma$ and coefficients $c_{i}$ such that
\begin{equation}
\sum_{i=0}^{3}c_{i}\frac{\mathrm{d}^{i}\zeta}{\mathrm{d}h^{i}}=\mathrm{d}v.\label{eq:lincombzeta}
\end{equation}
\end{lemma}
\begin{proof}
Observe that
\begin{eqnarray*}
w(z)(2h-z)^{\tfrac{5}{2}}\sum_{i=0}^{3}c_{i}\frac{\mathrm{d}^{i}\zeta}{\mathrm{d} h^{i}} 
& = & \left[-c_0z^3+ (6 c_0 h +c_1)z^2  + (-12 c_0 h^2 - 4 c_1 h + c_2)z \right.\\
 &  & \quad \left.+  ( 8c_0 h^3 + 4 c_1 h^2  -  2 c_2 h  +  3c_3)  \right] \mathrm{d} z \,.
\end{eqnarray*}
Now choose $v(z):=\dfrac{u}{(2h - z)^2} = \dfrac{w(z)}{(2h-z)^{\frac{3}{2}}}$, which is meromorphic on $\Gamma$, 
and has differential 
\[
w(z)(2h-z)^{\tfrac{5}{2}}\mathrm{d}v=\left[ (3 h + \tfrac12 \kappa)z^{2} + (2 \kappa h - 1)z - h \right]\mathrm{d}z.
\]
Equating the coefficients of the polynomials in $z$ and solving for $c_{i}$ yields the unique solution
\begin{eqnarray*}
c_{0} & = & 0\\
c_{1} & = & 3 h + \tfrac12 \kappa\\
c_{2} & = &  12 h^2 + 4 \kappa h - 1 \\
c_{3} & = & h( 4 h^2 + 2 h \kappa - 1) = \frac12\left[w( 2h )\right]^{2}
\end{eqnarray*}
and thus by construction we have proven the Lemma.
\end{proof}
Now we are ready to derive the linear and homogeneous Picard-Fuchs equation for the action $I(h)$:
\begin{theorem}
The scaled action $I(h)$ satisfies the Picard-Fuchs equation
\begin{equation}
\left[w(2 h)\right]^{2}I'''(h) + 2(12 h^2 + 4 \kappa h - 1 )I''(h)+(6 h +  \kappa )I'(h)=0 \label{eq:PFODEI}
\end{equation}
with the scaled energy $h$ as the independent variable.
\end{theorem}
\begin{proof}
To obtain the Picard-Fuchs ODE, perform a closed complex contour
integral to both sides of equation \eqref{eq:lincombzeta}. By definition
we have $\oint\zeta=4\pi I$, and given $v$ is meromorphic on $\Gamma$ the residues of $\mathrm{d}v$ are vanishing,
so that the right hand side gives $\oint\mathrm{d}v=0$ for any closed integration path.
\end{proof}

The fact that the Picard-Fuchs equation is of third order is related to the fact the curve is elliptic (genus $g=1$)
and that the differential $\zeta(h)$ is of third kind with a single pole. Moreover, the derivative of the residue of 
the pole with respect to $h$ vanishes, so that the first $h$-derivative of $\zeta$ is a differential of 2nd kind 
on a genus $g=1$ curve.
Now $w(z)$ is independent of $h$ and further derivatives do not create additional poles.
So the order of the Picard-Fuchs equation is $2g + 1 = 3$.

\subsection{Solving the Picard-Fuchs equation}

Clearly \eqref{eq:PFODEI} is an ODE in $I'$ thus $I=k_{3}$ is a
constant solution. To lower the order we introduce the scaled period $T(h) = 2\pi I'(h)$, 
which has the first kind differential $\frac{\mathrm{d} z }u$ on $\Gamma$.

The scaled period $T(h)$ hence satisfies the second order linear homogeneous ODE
\begin{equation}
T''(h)+ 2 \frac{12 h^2 + 4 \kappa h - 1}{\left[w(2 h)\right]^{2}}T'(h)+  \frac{6 h+\kappa}{ \left[w(2 h)\right]^{2}}T(h)=0.\label{eq:PFODE for T(nu)}
\end{equation}

It is interesting to observe that the leading coefficient
$c_{3}$ is proportional to $\left[w(2 h)\right]^{2}$, and thus in normalising the ODE 
the roots of $w(2h)$ given by $2h \in\left\{ 0,-\rho,\rho^{-1}\right\} $
become the (regular) singular points of the Picard-Fuchs ODE. 
Thus the partial fraction decomposition of the coefficient of $T'$ simply is 
$$2\left(\frac1{2h} +\frac1{2h + \rho} + \frac1{2h - \rho^{-1}}\right).$$

We are interested in series solutions at the singular point corresponding to the unstable
equilibrium, namely $h=0$. The general theory and procedure
for solving \eqref{eq:PFODE for T(nu)} via the method of Frobenius
can, e.g., be found in \cite{Boyce2001}.
We seek series solutions of the form
\[
\sum_{n=0}^{\infty}a_{n}(\varrho)h^{n+\varrho}
\]
where $\varrho$ is a root of the indicial equation. At the finite
singular points the indicial equation is $\varrho^{2}=0$.

\begin{remark}
We find that the Picard-Fuchs ODE has a regular singular point at
$h=\infty$, confirming it is of Fuchsian type. The indicial equation
for the singular point at infinity is
\[
\varrho_{\infty}^{2}-2\varrho_{\infty}+\tfrac{3}{4}=0
\]
with roots $\varrho_{\infty}=\tfrac{1}{2},\,\tfrac{3}{2}$ differing by an integer.
\end{remark}

By substitution of the Frobenius series into the ODE, the recursion
relation for $a_{n}(\varrho)$ is found to be
\begin{equation} 
a_{n}(\varrho)=\frac { 2n+2\varrho-1}{(n+\varrho)^{2}}  \left( \frac{\kappa}{2} (2n+2\varrho-1) a_{n-1}(\varrho) + (2n+2\varrho-3)a_{n-2}(\varrho) \right)\label{eq:a_n rho}
\end{equation}

Setting $\varrho=0$ and $a_{n}:=a_{n}(0)$ yields the Frobenius expansion of the regular solution
\[
T_r(h)=\sum_{n=0}^{\infty}a_{n}h^{n}.
\]
with coefficients obtained from \eqref{eq:a_n rho} at $\varrho = 0$ as
\begin{equation}
a_{n}=\frac{2n-1}{n^2} \left( \frac{\kappa}{2} (2n-1)a_{n-1} + (2n-3)a_{n-2}\right).\label{eq:a_n}
\end{equation}
We solve this second order recursion relation for $a_{n}$.
Without loss of generality, normalise the initial condition $a_{0}:=1$,
and we require that $a_{-1}:=0$. Thus we find that the next few coefficients
are
\begin{eqnarray*}
a_{1} & = & \frac{\kappa}{2},\\
a_{2} & = & \frac{3}{16}\left(3\kappa^{2}+4\right),\\
a_{3} & = & \frac{5}{32}\kappa\left(5\kappa^{2}+12\right),\\
a_{4} & = & \frac{35}{1024}\left(35\kappa^{4}+120\kappa^{2}+48\right),\\
a_{5}&=&\frac{63 \kappa}{2048}  \left(63 \kappa ^4+280 \kappa ^2+240\right)\,.
\end{eqnarray*}

\begin{theorem} \label{thm:an}
The recursion for $a_n$ is solved by
$$
    a_n =\frac{1}{4^n} \binom{2n}{n} \sum_{k=0}^{\lfloor \frac n2 \rfloor} \binom {2n - 2k} {k, n-k, n - 2k} \left( \frac{\kappa}{2} \right)^{n - 2k}\\
$$
where $\binom{n}{i,j,k} = \frac{n!}{i!j!k!}$ with $n = i+j+k$ is the trinomial coefficient. 
\end{theorem}

The proof of this theorem will be given later as a special case of the (approximate) solution of the more general recursion for $a_n(\varrho)$.

\begin{remark}
It is interesting to note that the sum $a_n$ can be summed to the hypergeometric function $_2F_1$
$$
a_n=8^{-n} \left(\frac{(2n)!}{(n!)^2}\right)^2 \kappa ^n\, _2F_1\left(- \frac{n-1}{2},-\frac{n}{2};-\frac{2n-1}{2};-\frac{4}{\kappa ^2}\right)
$$
which is always terminating because $n$ is an integer. 
It is well known that the complete elliptic integral of first kind $T$ can be expressed in terms of the hypergeometric function
as the function of the modulus of the elliptic curve, 
while here the coefficients of the Taylor series of $T_r(h)$ are given by the hypergeometric function.
\end{remark}

Since we have repeated indicial roots, we expect the second independent solution to be singular.
The general theory (see for example  \cite{Boyce2001}) says that the singular solution 
is of the form
\[
T_{s}(h):=T_{r}(h)\log h+\sum_{n=1}^{\infty}b_{n}h^{n}
\]
where
 $b_{n}:=\left.\tfrac{\mathrm{d}a_{n}(\varrho)}{\mathrm{d}\varrho}\right|_{\varrho=0}$.
The recursion relation for the $b_{n}$ at $\varrho = 0$ is thus given by
\begin{equation}
    b_{n}=\frac{(2n-1)\left(\kappa a_{n-1}+\tfrac{\kappa}2 n(2n-1)b_{n-1}+n(2n-3)b_{n-2}\right)+(8n-6)a_{n-2}}{n^{3}}.\label{eq:b_n}
\end{equation}
Along with the initial conditions on the $a_{n}$, we  also impose
that $b_{-1}:=0$ and $b_{0}:=0$. The first few coefficients are
\begin{eqnarray*}
b_{1} & = & \kappa,\\
b_{2} & = & \frac{1}{16}\left(21\kappa^{2}+20\right),\\
b_{3} & = & \frac{1}{96}\kappa\left(185\kappa^{2}+372\right),\\
b_{4} & = & \frac{1}{6144}\left(18655\kappa^{4}+56760\kappa^{2}+18672\right),\\
b_{5}&=&\frac{1}{20480}\kappa \left(102501 \kappa ^4+416360 \kappa ^2+313680\right) \,.
\end{eqnarray*}

We were not able to find an explicit solution for $a_n(\varrho)$. 
However, since we only need the derivative of $a_n(\varrho)$ at $\varrho = 0$ it 
is enough to find an approximate solution $\hat a_n(\varrho)$ that is valid up to terms of $\mathcal{O}(\varrho^2)$.
Using this we obtain an explicit formula for $b_n = a_n'(0) = \hat a_n'(0)$.

\begin{lemma}
The recursion for $a_n(\varrho)$ given by \eqref{eq:a_n rho} is approximately solved by 
\[
\hat a_n(\varrho) =2^n \kappa^n n! \frac{\left(\varrho+\frac12\right)_n}{(\varrho+1)_n^2} \sum_{k=0}^{ \lfloor \frac n2 \rfloor } 
      \frac{\left(\varrho+\frac12\right)_{n-k}}{(n-2k)! k!} \kappa^{n-2k}
\]
to leading order in $\varrho$, where $(x)_n$ is the Pochhammer symbol.
\end{lemma}
Note that this formula reduces to the explicit formula for $a_n$ given earlier 
when $\varrho = 0$ using the identity $\left(\frac12\right)_n = \frac{(2n-1)!}{(n-1)!2^{2n-1}}$.  
Hence the following proof will also prove Theorem~\ref{thm:an}.

\begin{proof}
We will show that $a_{n}(\varrho)=\hat{a}_{n}(\varrho)+\mathcal{O}(\varrho^{2})$
by induction. From the explicit recursion we find 
\[
a_{1}(\varrho)=\frac{(1+2\varrho)^{2}}{2(1+\varrho)^{2}}\kappa
\]
and 
\[
a_{2}(\varrho)=\frac{(1+2\varrho)(3+2\varrho)}{(2+\varrho)^{2}}+\frac{(1+2\varrho)^{2}(3+2\varrho)^{2}}{4(1+\varrho)^{2}(2+\varrho)^{2}}\kappa^{2}.
\]
One can easily check that $a_{1}(\varrho)=\hat{a}_{1}(\varrho)$,
Furthermore, the coefficients of $\kappa^{2}$ in $a_{2}(\varrho)$
and $\hat{a}_{2}(\varrho)$ coincide and the constant and linear coefficients
of $a_{2}(\varrho)$ and $\hat{a}_{2}(\varrho)$ are both 
\[
\tfrac{3}{4}+\tfrac{5}{4}\varrho+\mathcal{O}(\varrho^{2}).
\]
Now assume that the identity $a_{n}(\varrho)=\hat{a}_{n}(\varrho)+\mathcal{O}(\varrho^{2})$
holds for all $1\le n\le m-1$ for some fixed $2\le m\in\mathbb{N}$.
Now we are going to show that under these assumptions, $a_{m}(\varrho)=\hat{a}_{m}(\varrho)+\mathcal{O}(\varrho^{2})$
holds. 

We substitute the formulas $\hat{a}_{m-1}(\varrho)$
and $\hat{a}_{m-2}(\varrho)$ into the right hand side of \eqref{eq:a_n rho}
with $n=m$ and need to verify that $a_{m}(\varrho)$ so obtained
is equal to $\hat{a}_{m}(\varrho)+\mathcal{O}(\varrho^{2})$.

In order to simplify the recursion define $d_{m}:=a_{m}(\varrho)(2m+2\varrho+1)$ so
that 
\[
d_{m}=g_{m}(\tfrac{\kappa}{2}d_{m-1}+d_{m-2}),\quad\text{where }g_{m}:=\frac{(2m+2\varrho+1)(2m+2\varrho-1)}{(m+\varrho)^{2}}\,.
\]
Using the formula for $\hat{a}_{m}$ and the above definition of $d_{m}$
we define
\begin{equation}
\hat{d}_{m}:=\hat{a}_{m}(\varrho)(2m+2\varrho+1)=G_{m}(\varrho)\sum_{k=0}^{\left\lfloor \tfrac{m}{2}\right\rfloor }S_{m,m-2k}(\varrho)\kappa^{m-2k},
\label{eq:d_m_hat}
\end{equation}
where 
\[
G_{m}(\varrho):=\frac{2^{m+1}m!\left(\varrho+\tfrac{1}{2}\right)_{m+1}}{(\varrho+1)_{m}^{2}},\quad S_{i,j}(\varrho):=\frac{\left(\varrho+\tfrac{1}{2}\right)_{\tfrac{i+j}{2}}}{\left(\tfrac{i-j}{2}\right)!j!}\,.
\]
The equivalent claim is that $d_{m}  =  \hat{d}_{m}+\mathcal{O}(\varrho^{2})$, which implies
\begin{equation} \label{eq:d_m_recursion}
\hat{d}_{m} = g_{m}(\tfrac{\kappa}{2}\hat{d}_{m-1}+\hat{d}_{m-2})+\mathcal{O}(\varrho^{2}).
\end{equation}
Inserting the series \eqref{eq:d_m_hat} for $\hat{d}_m$ and collecting powers of $\kappa$
gives 
\[
G_{m}(\varrho)S_{m,m-2k}(\varrho)=g_{m}\left(\tfrac{1}{2}G_{m-1}(\varrho)S_{m-1,m-2k-1}(\varrho)+G_{m-2}(\varrho)S_{m-2,m-2k}(\varrho)\right)+\mathcal{O}(\varrho^{2})
\]
for $k = 0, \dots, \left\lfloor \tfrac{m}{2}\right\rfloor$.
Dividing this equation by $G_{m-1}(\varrho)S_{m-2,m-2k}(\varrho)$
and using the identities 
\[
\frac{G_{m}(\varrho)}{G_{m-1}(\varrho)}=\frac{m(2m+2\varrho+1)}{(m+\varrho)^{2}}=\frac{m}{2m+2\varrho-1}g_{m}
\]
and 
\[
\frac{S_{m,m-2k}(\varrho)}{S_{m-2,m-2k}(\varrho)}=\frac{2m-2k+2\varrho-1}{2k},\quad\frac{S_{m-1,m-2k-1}(\varrho)}{S_{m-2,m-2k}(\varrho)}=\frac{m-2k}{k},
\]
we have after some simplification
\[
\frac{m}{2m+2\varrho-1}=1-\frac{(m+\varrho-1)^{2}}{(m-1)(2m+2\varrho-1)}-\frac{\mathcal{O}\left(\varrho^{2}\right)}{G_{m-1}(\varrho)S_{m-2,m-2k}(\varrho)}.
\]
Multiplication by $(m-1)(2m+2\varrho-1)$  gives
\[
\varrho^{2}=-\frac{(m-1)(2m+2\varrho-1)}{G_{m-1}(\varrho)S_{m-2,m-2k}(\varrho)}\mathcal{O}\left(\varrho^{2}\right).
\]
The coefficient of $\mathcal{O}\left(\varrho^{2}\right)$ in this equation
evaluated at $\varrho=0$ is 
\[
-\frac{(m-1)(2m-1)}{G_{m-1}(0)S_{m-2,m-2k}(0)}=-\frac{\left(2m^{2}-3m+1\right)(k-1)!(m-1)!(m-2k)!}{2^{m}\left(\frac{1}{2}\right)_{m}\left(\frac{1}{2}\right)_{m-k-1}},
\]
which is non-zero for all  $k = 0, \dots, \left\lfloor \frac{m}{2}\right\rfloor$ and $m\ge2$,
and hence the the Taylor series of the factor multiplying $\mathcal{O}\left(\varrho^{2}\right)$
has a non-zero constant term. This shows that for each power of $\kappa$
in \eqref{eq:d_m_recursion} the estimation holds, and thus it holds for
the whole finite series. Hence by mathematical induction the Lemma is proved.
\end{proof}

With the previous Lemma it is now straightforward to find an explicit formula for the coefficients $b_n$.
\begin{theorem}
The recursion relation for $b_n$ is solved by 
\[
    b_n =\frac{1}{4^n} \binom{2n}{n} \sum_{k=0}^{\lfloor \frac n2 \rfloor} \binom {2n - 2k} {k, n-k, n - 2k} f_{n,k} \left( \frac{\kappa}{2} \right)^{n - 2k}\\
\]
where 
\[
   f_{n,k}:=  2 \mathrm O_n + 2 \mathrm O_{n-k} - 2 \mathrm H_n
\]
and $\mathrm H_n$ is the Harmonic number and $\mathrm O_n$ its odd cousin defined by 
\[
   \mathrm H_n := \sum_{k=1}^n \frac1k, \qquad  \mathrm O_n := \sum_{k=1}^n \frac{1}{2 k - 1} \,.
\]
\end{theorem}
\begin{proof}
Using the previous Lemma we can simply differentiate $\hat a_n(\varrho)$ and evaluate 
at 0 in order to get $b_n$. The derivative of the Pochhammer function is given 
in terms of the digamma function $\psi$, which for integer and half-integer values $n$ can be 
expressed in terms of the Euler-Mascheroni constant $\gamma$ and the Harmonic numbers,
$\psi(n) = -\gamma + \mathrm H_{n-1}$, where $\mathrm H_n = \sum_{k=1}^n \frac1k$. Using 
the recursion $\mathrm H_n = \mathrm H_{n-1} + \frac1n$ and $\mathrm H_{\frac12} = 2 - 2 \log 2$
the Harmonic number is thus also defined for half integers;
explicitly  $\mathrm H_{n-\frac12} = 2 \mathrm O_n + \mathrm H_{-\frac12}$.
Denote by $\hat a_n^k(\varrho)$ the coefficient of $\kappa^{n - 2k}$ in $\hat a_n(\varrho)$.
The logarithmic derivative of $\hat a_n^k(\varrho)$ at $\varrho = 0$ can thus be found as
\[
   \frac{{ a_n^k}{}'(0)}{ a_n^k(0)} = \frac{{\hat a_n^k}{}'(0)}{\hat a_n^k(0)} = \mathrm H_{n-\frac12} + \mathrm H_{n-k-\frac12} - 2 \mathrm H_n + 4 \log 2 = 2 \mathrm O_n + 2 \mathrm O_{n-k} - 2 \mathrm H_n
\]
and this determines the ``correction factor'' $f_{n,k}$ for the coefficient of $\kappa^{n-2k}$ in $b_n$.
\end{proof}

To obtain the solutions  $I_r(h)$ and $I_s(h)$ of the Picard-Fuchs equation 
we integrate $T_r(h)$ and $T_s(h)$ term-by-term, respectively, and get
\begin{subequations}\label{eq:IrIs}
\begin{eqnarray}
 I_{r}(h) & := &\frac1{2\pi} \int T_{r}(h)\mathrm{\, d}h=\frac1{2\pi} \sum_{n=0}^{\infty}\frac{a_{n}}{n+1}h^{n+1}\label{eq:I_r}\\
 I_{s}(h) & := & \frac1{2\pi} \int T_{s}(h)\,\mathrm{d}h\\
& = &  I_r(h) \log h + \frac1{2\pi} \sum_{n=0}^\infty \frac{1}{n+1}\left\{ b_n - \frac{a_n}{n+1}\right\} h^{n+1}
\end{eqnarray}
\end{subequations}
where the integration constants are fixed by the requirements that $I_r(0) = 0$ and $I_s(h) \to 0$ as $h\to0$.

\subsection{Particular action integrals}

Since \eqref{eq:PFODEI} is a linear third order equation, there must be
three linearly independent solutions. They are: the regular, singular,
and constant solutions. Thus the general solution is an arbitrary
linear combination of these, namely
\begin{equation}
I(h)  =  k_{1}I_{r}(h)+k_{2}I_{s}(h)+k_{3},\label{eq:arblincom}
\end{equation}
and upon differentiating 
\[
T(h)= k_1 T_{r}(h)+ k_2 T_{s}(h).
\]
We seek to find the $k_{i}$ that give the particular solutions
corresponding to the closed loops integrals along the paths $\beta_\pm$
as specified in Figure~\ref{fig:contours}.
The expansions obtained are normalised such that
$T_r = 1 + \mathcal O(h)$,
$2\pi I_r = h + \mathcal O(h^2)$,
$T_s = \log h + \mathcal O(h)$,
$2\pi I_s = h \log h + \mathcal O(h)$, 
and so the leading terms are 
$2\pi I(h) = 2\pi k_3 + k_2 h \log h + k_1 h + \mathcal O(h^2)$ and 
$T(h) =  k_1 +  k_2 \log h + \mathcal O(h)$.
Thus the constant $k_3$ is given by $I(0)$,
while both $k_1$ and $k_2$ are given by the leading order behaviour of $T(h)$.
When $T(h)$ is finite for $h\to 0$ then $k_2 = 0$ and $k_1$ is determined by $T(0)$.
Otherwise the leading order logarithmically diverging term and the constant term of $T(h)$ for small $h$
determine $k_1$ and $k_2$.

The four particular solutions to be found are given by the integrals
 $I_{\beta_{\pm}}$ and
 $T_{\beta_{\pm}}$.
In order to find the correct linear combinations we need to evaluate these integrals in the limit $h\rightarrow0$.

The $\beta$ integrals at $h=0$ are computed as real integrals.
The $\beta$ cycles are shrunk down, so that we are integrating along the real intervals 
$\Pi\left[\beta_+\right]:=\left(2h,\rho^{-1}\right)$ and
$\Pi\left[\beta_-\right]:=\left(-\rho,2h\right)$.
$I_{\beta_\pm}(0)$ is an elementary and finite real integral that gives
\begin{eqnarray*}
I_{\beta_{\pm}}(0) & = & \frac{1}{4\pi}  \oint_{\beta_{\pm}}\frac{ \mathrm{d}z}{\sqrt{1-\kappa z-z^2}} \\
& = & \frac{1}{2\pi} \int_{\Pi\left[\beta_{\pm}\right]}\frac{ \mathrm{d}x}{\sqrt{(x+\rho)(\rho^{-1}-x)}}=\frac1{\pi} \tan^{-1}\left(\rho^{\mp1}\right) \,.
\end{eqnarray*}
For the singular integral $T_{\beta_\pm}(h)$, the asymptotic behaviour for small $h$ is
\begin{eqnarray*}
T_{\beta_{\pm}}(h) & = & \oint_{\beta_{\pm}}\frac{\mathrm{d}z}{2w(z)\sqrt{2h-z}}=\int_{\Pi\left[\beta_{\pm}\right]}\frac{1}{\sqrt{x(x-2h)}}\frac{1}{\sqrt{1-\kappa x-x^2}}\,\mathrm{d}x\\
 & = &\pm\log\left(\pm h\right)\mp\frac12\log\left(\frac{64}{\kappa^2+4}\right)+\mathcal{O}(h).
\end{eqnarray*}
This can be shown as follows.
Define $\varphi(x):=[1-\kappa x-x^2]^{-\tfrac{1}{2}}=[(x+\rho)(\rho^{-1}-x)]^{-\tfrac{1}{2}}$
and split the integrals up as 
\[
T_{\beta_{\pm}}(h)=\int_{\Pi\left[\beta_{\pm}\right]}\frac{\varphi(x)-\varphi(0)}{\sqrt{x(x-2h)}}\,\mathrm{d}x+\int_{\Pi\left[\beta_{\pm}\right]}\frac{\varphi(0)}{\sqrt{x(x-2h)}}\,\mathrm{d}x.
\]
The first integral is a convergent elliptic integral, and when $h=0$ it becomes elementary. 
For $\beta_+$ it gives $-\log\left(\tfrac{4\rho^2}{\rho^{2}+1}\right)$, 
and for $\beta_-$ it gives $ \log \left(\tfrac{4}{\rho ^2+1}\right)$.
The second integral is divergent when $h\to0$ but elementary and can be integrated 
using hyperbolic trigonometric substitutions.
For $\beta_+$ it gives $  -\log\left(2\rho^{-1}\right)+\log h+\mathcal{O}(h)$, 
and  for $\beta_-$ it gives $ \log\left(2\rho\right)-\log(- h)+\mathcal{O}(h)$.
Adding the two integrals gives the stated result.

From these 4 integrals the coefficients $k_i$ can be determined as described above
and the result is
\begin{equation}
I_{\beta_{\pm}}  =\mp\frac12\log\left(\frac{64}{\kappa^2+4}\right)I_{r}\pm I_{s}+\frac1{\pi}\tan^{-1}\left(\rho^{\mp1}\right).\label{eq:Ibeta}
\end{equation}

\begin{remark}
Notice that the actions related to approaching the separatrix from either side satisfy $I_{\beta_+}(0)+I_{\beta_-}(0)=\tfrac12$,
corresponding to the fact that the total area of the symmetry reduced scaled phase space is $\pi$.
This area can also be seen as the residue at infinity of $I(h)$
  $$
2\pi\left(I_{\beta_+}(0)+I_{\beta_-}(0)\right)=\frac12 \cdot 2\pi \mathrm i \underset{\infty}{\text{Res}}\left\{ \zeta(h)\right\}=\pi.
  $$
As mentioned earlier the actual area (without discrete symmetry reduction) enclosed by a single connected contour  of $\PoissonHam$ is twice as large, and each connected component appears twice. 
Undoing the scaling then gives $4\pi \ell^2$ which is the area of the sphere $\mathcal{C}_\ell$ which is the 
($SO(3)$ reduced) phase space of the Euler top.
 \end{remark}

\section{The symplectic invariants}

Equipped with the Frobenius series expansions of the action integrals
obtained from the Picard-Fuchs ODE, we can now calculate the semi-global
symplectic invariants of the Euler top.

\subsection{Revisiting the Birkhoff normal form}
The Birkhoff normal form is a series for $h(J)$. From \eqref{eq:I_r} we instead have a series for 
the regular action in terms of the energy
$$2\pi I_r(h)=  h + \frac{\kappa}{4} h^2 + \frac{1}{16} \left(3 \kappa ^2+4\right)h^3 + \dots$$ 
The regular action can be obtained by 
 integrating $\zeta$ over the $\alpha$ cycles
\begin{equation}
 I_{\alpha}(h):=\frac{\mathrm i}{2\pi}\oint_{\alpha}\zeta(h)= h + \frac{\kappa}{4} h^2 + \frac{1}{16} \left(3 \kappa ^2+4\right)h^3 + \dots \label{eq:Ialpha}
\end{equation}
where the series can be obtained by Taylor expansion and taking residues.
We can omit the subscript $\pm$ for $\alpha$ since the two cases yield the same series expansion. 
The relation between the action integral and the regular Frobenius series thus is
$$
I_{\alpha}(h)= 2\pi I_r(h).
$$
Thus by inverting this series we recover the Birkhoff normal form, so that we can identify $ I_\alpha = J$.
The fact that the Birkhoff normal form at a hyperbolic point is given 
by the integral of the $\alpha$ cycles is a general phenomenon, see \cite{Dullin2013} for a general proof.
The  idea is that this works in a way similar to an elliptic point. Near such a point the action is given by a closed 
loop integral over a periodic orbit which depends on the energy. Thus the action is obtained as a function of the energy
and inverting this function gives the energy as a function of the action. Now the Birkhoff normal form  is  a form of 
the Hamiltonian that depends on a single variable only and by uniqueness of this function we can identify it with the 
inverse of the action function. A similar type of argument works near a hyperbolic point, see \cite{Dullin2013} for the details.

\subsection{The semi-global symplectic invariant}

The semi-global symplectic invariant 
 $\sigma(J)$ is the power series given by  
the regular part of the composition of the singular action integral with the inverse of
the regular action integral. More precisely there are two cases, depending on the sign of $h$:
\begin{equation}
 2\pi \left(I_{\beta_{\pm}}\circ  I_{\alpha}^{-1}\right)(J)  =   \mathcal A _{\pm} \pm J\log\left( \pm J\right) \mp J \mp \sigma(J)  .
\label{eq:IbIa}
\end{equation}
Here $\mathcal A _\pm$ is the area enclosed by the separatrix after discrete symmetry reduction, 
$\mathcal A _{\pm} =2 \pi I_{\beta_{\pm}}(0) = 2 \tan^{-1}\left( \rho^{\mp1}\right)$ 
so that $\mathcal A _+ + \mathcal A _- = \pi$.
Thus we obtain the semi-global symplectic invariant $\sigma$:
\begin{theorem}
The semi-global symplectic invariant of the Euler top with distinct moments of inertia 
reduced by discrete symmetry near the hyperbolic equilibrium
is defined via \eqref{eq:IbIa} and is given by
\begin{eqnarray*}
\sigma(J) & = &\frac12\log\left(\frac{64}{\kappa^2+4}\right)J
-\frac{3\kappa}{8}J^{2}
-\frac{15\kappa^{2}+32}{96}J^{3}
 -\frac{5\kappa\left(11\kappa^{2}+36\right)}{512}J^{4}
 \\ &&
 -\frac{945\kappa^{4}+4200\kappa^{2}+2672}{10240}J^{5}
 -\frac{7\kappa\left(527\kappa^{4}+2960\kappa^{2}+3600\right)}{40960}J^{6}
 \\&&
 -\frac{65709\kappa^{6}+446040\kappa^{4}+801360\kappa^{2}+241664}{688128}J^{7}+\mathcal{O}\left(J^8\right).
\end{eqnarray*}
\end{theorem}
\begin{proof}
The action $I_\beta(h)$ is given as the series expansions  \eqref{eq:Ibeta}, \eqref{eq:IrIs} whose coefficients where obtained
from the Frobenius expansion of the Picard-Fuchs equation. The action $ I_\alpha(h)$ is similarly given by \eqref{eq:Ialpha},
and the inverse of this series is the Birkhoff normal form $H^*(J)$. Composing $I_\beta(h)$ with $H^*(J)$ gives 
the series of the action $I_\beta$ in terms of the normal form action $J$, from which 
the symplectic invariant $\sigma(J)$ can be read off using the definition in \eqref{eq:IbIa}.
\end{proof}

Notice that for $\kappa = 0$ (i.e.{} $\rho = 1$, i.e.{} $\Theta_1^{-1} - \Theta_2^{-1} = \Theta_2^{-1} - \Theta_3^{-1}$)
the invariant is an even function of $h$, so that both sides of the separatrix, for positive 
and negative $h$ or $J$ are the same. 
The linear term has maximal value of
$\log 4$ at $\kappa = 0$. For positive $\kappa$ all higher order terms are negative.
The series expansion of $\sigma$ has been numerically verified and agrees well with 
the values obtained from a direct numerical computation of $I_{\beta_\pm} \circ  I_\alpha^{-1}$. 

\section{Convergence of the symplectic invariant}

Convergence of the Birkhoff normal form for analytic integrable systems has recently been proven by Zung \cite{Zung2005},
also the references therein for earlier less general results. Convergence of the series expansions of the actions themselves 
are classical, and can also be obtained from the Picard-Fuchs equation. However, we are not aware of results 
about the convergence of the symplectic invariant. In general it is considered to be a formal series only. However,
in the analytic case one may expect that the symplectic invariant has non-zero radius of convergence.
Here we briefly report that numerically we find that the symplectic invariant has the same non-zero radius of convergence 
as the Birkhoff normal form; the reason for this remarkable observation is unclear.

To analyse the asymptotics of $a_n$ define the ratio
$r_n = \frac{(2n+1) a_n }{ (2n-1) a_{n-1}}$ and the recursion becomes 
$r_n = \frac{4 - n^{-2}}{\tfrac{\kappa}2 + r_{n-1}^{-1}}$.
The leading order iteration $r_n = 2 \kappa + \tfrac4{r_{n-1}}$ has two fixed points
at $r = -2\rho^{-1}$ and $r = 2 \rho$. The positive fixed point $r = 2 \rho$ is stable 
for positive $\kappa$ (or $\rho > 1$), while the negative fixed point $r = -2\rho^{-1}$ 
is stable for negative $\kappa$ (or $\rho < 1$).
The radius of convergence in $h$ is given by $|r_\infty|^{-1}$ which gives $\tfrac12 \min( \rho, \rho^{-1})$,
and is thus controlled by the roots of $w^2$ closest to zero.
For $b_n$ a similar argument \label{argument for convergence of b_n} works after explicitly controlling the size of $\frac{a_{n-1}}{b_{n-1}}$ and  $\frac{a_{n-2}}{b_{n-1}}$.
Thus both series with coefficients $a_n$ and $b_n$ have the same radius of convergence $\tfrac12 \min( \rho, \rho^{-1})$.

The Birkhoff normal form is given by the inverse of the series with coefficients $\frac{a_n}{n+1}$ 
which converges, and thus this series converges as well, as expected from the general theory.
However, since inverting a series is a highly non-linear process explicit formulas
cannot easily be obtained. Similarly, the symplectic invariant is based on the convergent
series with coefficients $\frac{b_n}{n+1} - \frac{a_n}{{(n+1)}^2}$ composed with the Birkhoff normal 
form, so again we expect convergence.
Numerically computing the ratios of successive coefficients 
indicates that the radii of 
convergence of both the Birkhoff normal form and the symplectic invariant are equal. This is a surprising observation. Furthermore, observing the morphology and rates of decay of ratios of coefficients for varying $\kappa$ and increasing $N$ leads us to conclude that the radius of convergence of the Birkhoff normal form and the symplectic invariant is at least 25\% larger than the radius of convergence of $I_\alpha$ and the regular part of $I_\beta$.
For more details on these results see \cite{Papadopoulos2013}.

\section{Non-equivalence with the pendulum}

To compare the Euler top and the pendulum the
topologies of their separatrices
need to be made the same by discrete symmetry reduction, 
as is required by Theorem \ref{thm:equivalence}. 
Initially they are not the same, two joined circles intersecting twice versus a ``figure-eight''.
The discrete reduction for the top was described in section 2 and illustrated in 
Figure~\ref{fig:ActionAreas}. For the pendulum a discrete symmetry reduction 
reduces it to a similar phase portrait.
The semi-global
symplectic invariant near the unstable hyperbolic equilibrium of the
pendulum is given in \cite{Dullin2013} to leading order as
\[
\sigma_{\mathrm P}(J)=\ln32\, J+\mathcal{O}\left(J^{2}\right).
\]
For the pendulum to be equivalent to a particular Euler top the 
leading term in
\[
\sigma_{\mathrm E}(J)=\frac{1}{2}\ln\left(\frac{64}{\kappa^{2}+4}\right)J+\mathcal{O}\left(J^{2}\right)
\]
would need to coincide. However, the maximum value  of 
the leading order term which is attained for $\kappa = 0$ is $\ln 4$, 
so that the leading order term for the pendulum is always bigger
than that of any Euler top. Hence there is no Euler top that is semi-globally 
equivalent to the pendulum.

This may seem to contradict a theorem in \cite{Holm1991,Marsden1999},
where it is shown that rigid-body motion reduces 
to pendulum motion when using a different Poisson structure for the rigid body.
However, only the Poisson structure $\PoissonMat$ comes from the original physical system,
and in our notion of equivalence we are not allowed to change this Poisson structure.


\bibliography{bib}
\bibliographystyle{plain}

\end{document}